\documentclass[12pt,a4paper]{amsart}
\usepackage[english]{babel}
\usepackage[utf8]{inputenc}
\usepackage{amsmath,amsfonts,amssymb,amsthm}
\usepackage{a4wide}
\usepackage{hyperref}
\usepackage{graphicx}
\usepackage{color}

\newtheorem{thm}{Theorem}[section]
\newtheorem{lem}[thm]{Lemma}
\newtheorem{cor}[thm]{Corollary}
\newtheorem{fact}{Fact}
\newtheorem{prop}[thm]{Proposition}
\theoremstyle{definition}

\newtheorem{exam}[thm]{Example}
\theoremstyle{remark}
\newtheorem*{rem}{Remark}

\newcommand{\bd}{\stackrel{\rm bd}{\sim}}

\DeclareMathOperator{\dens}{dens}

\newcommand{\N}{\mathbb{N}}

\newcommand{\R}{\mathbb{R}}

\newcommand{\X}{\mathbb{X}}
\newcommand{\Z}{\mathbb{Z}}

\renewcommand{\phi}{\varphi}
\renewcommand{\rho}{\varrho}
\renewcommand{\epsilon}{\varepsilon}

\begin{document}

\title[BD classes in repetitive hulls]{Number of bounded distance
equivalence classes in hulls of repetitive Delone sets}

\author{Dirk Frettl\"oh}
\address{Technische Fakult\"at, Bielefeld University}
\email{dirk.frettloeh@math.uni-bielefeld.de}

\author{Alexey~Garber}
\address{School of Mathematical \& Statistical Sciences, The University of Texas Rio Grande Valley, 
1 West University Blvd., Brownsville, TX 78520, USA.}
\email{alexey.garber@utrgv.edu}

\author{Lorenzo Sadun}
\address{Department of Mathematics, University of Texas, 2515 Speedway, PMA 8.100
Austin, TX 78712, USA}
\email{sadun@math.utexas.edu}

\date{\today}

\begin{abstract}
Two Delone sets are bounded distance equivalent to each other if there
is a bijection between them such that the distance of corresponding 
points is uniformly bounded. Bounded distance equivalence is an
equivalence relation. We show that the hull of a repetitive Delone set 
with finite local complexity has either one equivalence class or
uncountably many.
\end{abstract}

\maketitle

\section{Introduction} \label{sec:intro}
Delone sets are central objects of study in the theory of aperiodic order 
and give rise to dynamical systems and topological objects with interesting
properties \cite{BG}.  
A \emph{Delone set} $\Lambda \subset \R^d$ is a set that is both \emph{uniformly
discrete} (that is, there is $r>0$ such that each open ball of radius $r$ contains
at most one point of $\Lambda$) and \emph{relatively dense} (that is,
there is $R>0$ such that each closed ball of radius $R$ contains at least one point 
of $\Lambda$). A \emph{point lattice} is the $\Z$-span $\langle v_1, \ldots,
v_d \rangle_{\Z}$ of $d$ linear independent vectors in $\R^d$. Each point 
lattice is a Delone set. 
A Delone set $\Lambda$ in $\R^d$ is \emph{$d$-periodic} if the set 
\[ P_{\Lambda} = \{ t \in \R^d \mid t+\Lambda = \Lambda \} \]
of its period vectors is a point lattice. Of course for every point lattice 
$\Lambda$ we have $\Lambda = P_{\Lambda}$, hence each point lattice
is $d$-periodic. A Delone $\Lambda$ set is \emph{nonperiodic} if 
$P_{\Lambda} = \{ 0 \}$. 

Two Delone sets $\Lambda, \Lambda'$ in $\R^d$ are called 
\emph{bounded distance equivalent} ($\Lambda \bd \Lambda'$) or \emph{bde}, if there 
is a bijection $\phi \colon  \Lambda \to \Lambda'$ such that $|x - \phi(x)|$ 
is uniformly bounded. It is easy to see that $\bd$ is an equivalence relation
for Delone sets.

In the 1990s several authors studied the question of whether a given Delone
set in $\R^d$ is bounded distant equivalent to a point lattice
\cite{DSS,DO1,HZ}.
In \cite{DO1} it was shown that any two point lattices in $\R^d$ with the
same density are bde. It is a simple consequence that any two $d$-periodic Delone sets
in $\R^d$ with the same density are bde. This leads to considering aperiodic
Delone sets, for instance the vertices of a Penrose tiling. There are two well 
studied classes of aperiodic Delone sets: cut-and-project sets and Delone 
sets arising from substitution tilings. For details and a precise definition
of an aperiodic Delone set, see \cite{BG}. 
%In particular, any aperiodic Delone set is nonperiodic. 
Recently, the question of whether an aperiodic Delone set is bde
to a point lattice has gained some interest. For the bd-equivalence of
cut-and-project sets to a point lattice see \cite{DO2,FG,H,HK,HKK}.
The bd-equivalence of Delone sets from substitution tilings  to a point lattice
was studied in \cite{ACG,FG,FSS,HZ,Solo1, Solo3}. Since the bd-equivalence of
Delone sets to a point lattice is well understood, the focus now turns 
to the question of when two Delone sets are bde.

This is where hulls of nonperiodic Delone sets come into play.
In the theory of aperiodic order  the \emph{hull} $\X_{\Lambda}$ of a 
given Delone set $\Lambda$ is a central object of interest. 
The hull of $\Lambda$ is the closure of the orbit 
of $\Lambda$ 
under translations in the topology defined by the ``big box'' metric (aka the Gromov-Hausdorff topology).\footnote{Some authors work
with a different topology, called the ``local rubber'' or ``Chabauty-Fell'' topology. When the
Delone set has finite local complexity, defined below, then the two 
topologies agree.}
Specifically, the distance between two $(r,R)$-Delone sets $\Lambda$ and 
$\Lambda'$ is the infimum over all $\varepsilon  \in (0,1)$ such that there
exist $x, x' \in \R^d$, each of size $\varepsilon/2$ or less, such that 
$\Lambda-x$ and $\Lambda'-x'$ agree exactly on $B_{1/\varepsilon}$. If no
such $\varepsilon$ exists, then $d(\Lambda, \Lambda')=1$.

The hull may be studied as a topological object \cite{AP,KP,Sbook},
or as a dynamical system $(\X_{\Lambda}, \R^d)$ \cite{BG}, where $\R^d$ 
stands for the action of translations by $x \in \R^d$. 

This raises the question ``How many bde  classes 
does $\X_{\Lambda}$ contain?''. A partial answer was given in \cite{Solo2}:
If $\Lambda$ comes from a substitution tiling meeting some technical conditions, then 
$\X_{\Lambda}$ consists
of uncountably many bde classes. Here we provide a generalization that covers all 
Delone sets that are repetitive and have finite local complexity (see definitions below). 
This category includes the vast majority of the substitution tilings and cut-and-project sets 
that have been studied to date.

Two geometric properties of the Delone set $\Lambda$ are closely associated with 
dynamical properties of $(\X_{\Lambda},\R^d)$. 
A \emph{patch} in a Delone set $\Lambda$ is a set $\Lambda \cap K$
for some compact $K \subset \R^d$. A Delone set $\Lambda$ has  
\emph{finite local complexity} (FLC)
if for any compact set $K \subset \R^d$ there are only finitely many different 
patches $(\Lambda-x) \cap K$ $(x \in \R^d)$, up to translation. 
A Delone set $\Lambda$ is called \emph{repetitive} if for each compact set 
$K \subset \R^d$ such that $K \cap \Lambda$ is non-empty, the set 
\[ \{ x \in \R^d \mid (\Lambda-x) \cap K = \Lambda \cap K \} \]
is a Delone set. The uniform density radius $R$ of this Delone set is the 
\emph{repetitivity radius}  of the patch $\Lambda \cap K$. 
The following fact is the essence of the work of several 
authors, see \cite{BG, Sbook} for details.

\begin{fact}
Let $\Lambda \subset \R^d$ be a Delone set. Then $\X_\Lambda$ is compact if 
and only if $\Lambda$ has FLC.
If $\Lambda$ has FLC, then the dynamical system $(\X_{\Lambda}, \R^d)$ 
is minimal (that is, each orbit is dense) if and only if $\Lambda$ is repetitive.
\end{fact}

Now we can state our main result.

\begin{thm}\label{thm:one-or-many}
Let $\Lambda$ be a repetitive Delone set in $\R^d$ having FLC and such that the 
density of $\Lambda$ exists.
Let $\X_{\Lambda}$ be the hull of $\Lambda$. Then
$\X_{\Lambda}$ either consists of a single bde class or  $\X_{\Lambda}$
contains uncountably many bde classes.
\end{thm}
In fact we show that the number of bde classes is either 1 or $2^{\aleph_0}$,
where $\aleph_0$ is the cardinality of $\Z$. It is clear that the number of bde classes
can't be larger than $2^{\aleph_0}$ \cite{G,SmSo}, insofar as there are only $2^{\aleph_0}$ elements of 
$\X_\Lambda$, up to translation. 

During the completion of this paper, Smilansky and Solomon released a
preprint \cite{SmSo} in which a stronger version of this result is proven using
the language of dynamical systems. Their theorem, which was developed completely independently
from ours, does not require FLC and 
considers the orbit closure of $\Lambda$ in the local rubber topology. They replace 
our assumption of repetitivity with one of minimality. They also do not require 
that the density of $\Lambda$ exists. The cost of that generality is 
complexity, in that their proof is longer and considerably more technical than ours. 

An interesting consequence of both results is the following: 
In \cite{FG} it was shown that a certain one-dimensional collection
of Delone sets, namely, the set of cut-and-project sets using 
half of the window of the famous Fibonacci tiling \cite{BG}, has
at least two different bde classes. Theorem 
\ref{thm:one-or-many} then yields the following result.

\begin{cor}
The hull $\X_{HF}$ of the 'Half-Fibonacci' cut-and-project tiling 
contains uncountably many bde classes.
\end{cor}

\section{Auxiliary results}

Let us fix some more notation. In the sequel, let $\#M$ denote the cardinality of 
a (typically finite) set $M$. We denote the Euclidean norm of $x \in \R^d$ 
by $\|x\|$. The closed ball of radius $r$ about $x$ is denoted by $B_r(x)$. 
The $d$-dimensional Lebesgue measure  of a set $A \subset \R^d$ is 
denoted by $\mu(A)$ and all sets in the paper are compact and measurable unless noted otherwise. Let $A^{+\varepsilon}$ denote the $\varepsilon$-tube
of the boundary of $A$. That is, 
\[ A^{+\varepsilon} = \{ x \in \R^d \mid d_2(x,\partial A) \le \varepsilon \}, \]
where $\partial A$ is the topological boundary of $A$ and $d_2$ is the standard Euclidean distance in $\R^d$. Hence $\mu(A^{+1})$ 
is the Lebesgue measure of the set of all points whose distance 
to the boundary of $A$ is one or less.

Any $d$-periodic Delone set has a well defined density, in the sense
of ``average number of points per unit volume''. The same holds for cut-and-project sets and for Delone sets from primitive substitutions.
In general, the definition of the density of a Delone set can be
tricky. It is easy to construct Delone sets having no well defined density
(for instance $-\N \cup 2\N$ in $\R$). There are even repetitive
Delone sets without density, see \cite[Thm. 5.1]{LP}, or \cite{FS}
for a simpler example.

In order to define the density of an arbitrary Delone sets we need
van Hove sequences. A \emph{van Hove sequence} is a sequence
$(A_i)_i$ of compact subsets of $\R^d$ such that for all $\varepsilon>0$
\begin{equation} \label{eq:def-vanhove}
\lim_{i \to \infty} \frac{\mu(A_i^{+\varepsilon})}{\mu(A_i)} =0. 
\end{equation}
A Delone $\Lambda$ set has density $\dens(\Lambda)$ if for
all van Hove sequences $(A_i)_i$ the limits 
\[ \lim_{i \to \infty} \frac{\#(A_i \cap \Lambda)}{\mu(A_i)} \]
exist and are identical. In that case $\dens(\Lambda)$ is the value of these limits.
An important tool in this context is the following result. In the context of
this paper it translates as follows.

\begin{thm}[\cite{Lac}] \label{thm:lacz}
Let $\Lambda$ be a Delone set in $\R^d$ with density $\dens(\Lambda)$.
$\Lambda$ is bde to some lattice in $\R^d$ if and only if there is $c>0$ 
such that, for all bounded measurable sets $E \subset \R^d$, 
\[ \big| \#(\Lambda \cap E) - \dens(\Lambda) \mu(E) \big| 
\le c \mu(E^{+1}) \]
\end{thm}
The proof of this result relies on the infinite version of 
the Hall Marriage Theorem \cite{Rado}. The same arguments can be used 
not only to compare a Delone set $\Lambda$ with $\alpha \Z^d$, 
but to compare two arbitrary Delone sets $\Lambda, \Lambda'$ 
as well. 
\begin{thm}[\cite{FSS}] \label{thm:l-not-bd}
Let $\Lambda, \Lambda'$ be two Delone sets in $\R^d$.  
Suppose there is a van Hove sequence $(A_i)_i$ such that 
\begin{equation}\label{eq:non_BD_condition}
\lim_{n\to\infty} 
\frac{| \#(\Lambda \cap A_i) - \#(\Lambda' \cap A_i) |}{\mu(A^{+1}_i)} 
= \infty,
\end{equation}   
then $\Lambda \not\bd \Lambda'$.
\end{thm}

In \cite{FSS}, the result above was formulated using sets $A_i$ that are
unions of lattice cubes of appropriate size. Since we deal with 
Delone sets we can ``approximate'' any van Hove sequence $(A_i)_i$
by an appropriate union of lattice cubes that are small enough that each cube
contains either one or zero points of $\Lambda$. For more details
see \cite{G} or \cite{FSS}.
 
\begin{exam} \label{ex:onlyrich}
As a test case for what follows we consider the set  
$L = \Z \cup \{ 1/2 + 2^n \mid n \in \N_0 \}$ and compare $L$ to the integers $\Z$. 
Both have density 1, but $\Z \not \bd L$. This can be seen by using 
Theorem \ref{thm:lacz} above, and observing for each $k>0$ there are intervals 
$Q_i : = [0,2^i+1]$ such that 
\[ \# (Q_i \cap \Lambda) - \#(Q_i \cap \Z) > i. \]
Note that $L$ contains ``rich'' regions (i.e. regions where
the number of points is larger than the expected number of points according to the density), but no ``poor'' 
regions. 
\end{exam}

\section{Proof of Theorem \ref{thm:one-or-many}}

Note the difference between a patch $E \cap \Lambda$, which is a finite collection of points,
and its support $E$. Also note that we have two different ways to move
patches around: either by translating both its support and $\Lambda$ to get
$(E \cap \Lambda)-x = (E-x) \cap (\Lambda-x)$, 
or by translating just the support to get $(E-x) \cap \Lambda$. 
Both ways yield patches occurring in some $\Lambda' \in \X_{\Lambda}$ and we will use both.

The strategy of the proof of Theorem \ref{thm:one-or-many} is as follows. If $\Lambda$ is not bde to a 
lattice, then there exist large regions that are ``deviant'' (meaning either very rich or very poor in points of $\Lambda$ compared to the expected number of points according to the density of $\Lambda$) and translates of those
regions that are either ``normal'' (neither rich nor poor) or have a discrepancy of the opposite sign of the 
original region; we will refer to the latter sets as ``normal'' anyway. We will recursively define larger and larger 
regions $P_i$ around the origin and find elements of $\X_\Lambda$ for which the regions $P_i$ are either deviant
or normal. In fact, for each infinite word in $\{D,N\}^\N$ we will construct an element of $\X_\Lambda$ 
where the form of each $P_i$ corresponds to the $i$-th letter in the word. 
Any two such elements of $\X_{\Lambda}$ with words $u,u'$ that differ in infinitely
many letters are not bde to each other by Theorem \ref{thm:l-not-bd}. Since there are uncountably many
sequences in $\{D,N\}^\N$, and since each tail-equivalence class is countable, there are uncountably many 
elements of $\X_\Lambda$, no two of which are bde. 

We construct deviant and normal patches $D_i'$ and $N_i'$ around the origin recursively. 
Let $\rho$ be the density of $\Lambda$ and suppose that $D_{i-1}'$ and $N_{i-1}'$ have already been 
constructed. Since $\Lambda$ is not bde to a lattice, there exist 
a region $D_i'$ such that the ratio
$$ \frac{|\#(\Lambda \cap D_i') -\rho \mu(D_i')|}{\mu(D_i^{+1})}$$
is arbitrarily large. In other words, for which $D_i'$ is deviant. As shown below, there is also a ``normal'' 
translate $N_i'$ of $D_i'$ for which the discrepancy $\#(N_1' \cap \Lambda) - \varrho \mu(N_i')$
has the opposite sign as the discrepancy of $D_i'$. Without loss of generality, we can choose $D_i'$
so large that there are copies of $D_{i-1}'$ and $N_{i-1}'$ near the center of $D_i'$, and likewise near
the center of $N_i'$. Note that the construction so far only involves looking for regions in the fixed Delone
set $\Lambda$. 

Next we construct Delone sets in $\X_\Lambda$ corresponding to each infinite word in $\{D,N\}^\N$. Let $u$ be 
such a word. Place a copy of $D_1'$ or $N_1'$, centered at the origin, according to whether the first letter of $u$ is
$D$ or $N$. Call this copy $D_1$ or $N_1$, and let $P_1 = D_1$ or $N_1$. 
Since both $D_2'$ and $N_2'$ contain copies of both $D_1'$ and $N_1'$ near their centers, we can 
extend $P_1$ to a copy $D_2$ or $N_2$ of $D_2'$ or $N_2'$, according to whether the second letter of 
$u$ is $D$ or $N$, and we can call this copy $P_2$. Repeat for each index $i \in \N$. Note that each $P_i$ is a patch
$N'_i-x_i$ or $D'_i-x_i$ in $\Lambda - x$. The union of the $P_i$'s is a Delone set $\Lambda_u \in \X_\Lambda$. 
Specifically, $$ \Lambda_u = \bigcup_i P_i = \lim_{i \to \infty} (\Lambda - x_i).$$

The tricky point is that, if $u \ne u'$, then the $i$-th patch $P_i$ of $\Lambda_u$ is not perfectly aligned with 
the $P_i$ of $\Lambda_{u'}$, so we cannot directly compare the two $P_i$'s in Theorem \ref{thm:l-not-bd}.
Instead, we must compare patches defined by small translates (within $\Lambda_u$) of supports of deviant patches $D_i$ in $\Lambda_u$ to 
normal patches $N_i$ in $\Lambda_{u'}$, or vice-versa. 
The remainder of the proof is a series of estimates to show that suitable $D'_i$'s and $N'_i$'s exists, such that
these small translates are still sufficiently deviant to apply Theorem \ref{thm:l-not-bd}.

The following result is standard, but for completeness we provide a sketch of the proof.
\begin{lem}
Let $\Lambda$ be a Delone set in $\R^d$ with density $\varrho$ and let $E$ be a 
bounded measurable set. Then there exist vectors $x_1,x_2\in\R^d$ such that 
\[ \#((E+x_1)\cap \Lambda)\leq \varrho\mu(E) \qquad \text{and} \qquad \#((E+x_2)\cap 
\Lambda)\geq \varrho\mu(E).\]
\end{lem}

\begin{proof}[Sketch of proof:]

If we average $\#((E+x) \cap \Lambda)$ over all values of $x \in \R^d$ we must 
get $\varrho\mu(E)$. This means that
for every $\epsilon>0$ there must a point $x$ for which $\#((E+x)\cap \Lambda)\le \varrho\mu(E) + \epsilon$. 
However, $\#((E+x)\cap \Lambda)$ is always an integer, and so for small enough $\epsilon$ cannot be strictly between
$\varrho\mu(E)$ and $\varrho\mu(E)+\epsilon$. Thus there must be an $x_1$ satisfying the first inequality. The
second is similar.
\end{proof}

The following result is then immediate and allows us to construct normal patches $N'_i$ from deviant patches $D'_i$.

\begin{lem} \label{lem:exceptional}
Let $\Lambda \subset \R^d$ be a Delone set with density $\varrho$
such that $\Lambda$ is not bde to any lattice in $\R^d$. Let $E$ be a
compact subset of $\R^d$ such that $\#(\Lambda \cap E)-\varrho
\mu(E)\geq 0$ ({\em resp.} $\leq 0$). Then there is a translation $E-x$ of $E$ such that
$\#(\Lambda \cap (E-x))-\varrho \mu(E)\leq 0$ ({\em resp.} $\geq 0$).
\end{lem}

We also need the following elementary identity.

\begin{lem} \label{lem:r+l}
Let $r,\ell>0$, then $(E^{+\ell})^{+r}\subseteq E^{+(\ell+r)}$.
\end{lem}
\begin{proof}
Let $z \in (E^{+\ell})^{+r}$. By definition there are $y \in \partial E$ and $y'$
such that $\| z - y'\| \le r$ and $\|y' - y\| \le \ell$, hence 
\[ \|z - y\| \le \| z - y' \| + \| y'-y\| \le \ell+r, \]
hence $z \in E^{+(\ell+r)}$. 
\end{proof}

We also need some result on estimates on the minimal and the maximal
number of points of a Delone set $\Lambda$ within a region $E$. In the
sequel we always assume that $E$ is a bounded measurable set.
\begin{lem} \label{lem:r-R-estim} There exist constants $\eta, \eta'$ such that
for any Delone set $\Lambda$ with parameters $r>0$ and $R>0$ and any bounded region $E \subset \R^d$,  
\[ \#(E \cap \Lambda) \ge \frac{\eta}{R^d} (\mu(E) - \mu(E^{+R}))
\quad \mbox{ and } \quad
\#(E \cap \Lambda) \le \frac{\eta'}{r^d} (\mu(E) + \mu(E^{+r})) \]
Here, $\eta$ and $\eta'$ depend only on $r$, $R$, and the dimension $d$.
\end{lem}
\begin{proof} 
Consider a periodic packing of $\R^d$ by balls of radius $R$. (For instance, we could tile 
$\R^d$ by cubes of side length $2R$ and place one ball in each cube.) Averaging over translates of 
this packing, the number of balls with center in $E - E^{+R}$ is the packing density (which is a constant
$\eta$ that depends on dimension divided by $R^d$) times the volume of $E - E^{+R}$. Therefore there exists a 
specific packing where the number of such balls is at least $\frac{\eta}{R^d} \mu(E-E^{+R})$, which in turn is 
at least $\frac{\eta}{R^d}(\mu(E) - \mu(E^{+R}))$. Each of these disjoint balls is contained completely in $E$, and by the
Delone property each contains at least one point of $\Lambda$. This establishes the first inequality.

The second inequality is similar, except that we use a covering by balls of radius $r$ (e.g. by starting with 
a tiling of $\R^d$ by cubes of side length $2r/\sqrt{d}$ and using circumscribed balls) instead of a packing
by balls of radius $R$. Each point in $E \cap \Lambda$ must lie in a ball whose center is either in $E$ or in
$E^{+r}$. However, each such ball can contain at most one point in $\Lambda$ and the number of such balls is 
bounded by $\frac{\eta'}{r^n} ( \mu(E\cup E^{+r}) \le \frac{\eta'}{r^n} (\mu(E) + \mu(E^{+r}))$. \end{proof}

\begin{lem} \label{lem:thickboundary}
For all $\ell \ge 1$, $\mu(E^{+\ell}) \le \ell^d \mu(E^{+1})$.
\end{lem}
\begin{proof}
First we prove  $\mu((\frac{1}{\ell} E)^{+1})\leq \mu(E^{+1})$. Let $\epsilon>0$. Let $x_1,\ldots,x_n\in \partial(\frac{1}{\ell}E)$ be points such that the set of $\epsilon$-balls centered at $x_i$ covers $\partial(\frac{1}{\ell}E)$. 

The balls of radii $1+\epsilon$ centered at $x_i$'s cover $(\frac{1}{\ell}E)^{+1}$; let $X$ be the union of these balls. 
Similarly, the balls of radii $1+\epsilon$ centered at $\ell x_i$'s are contained in $E^{+(1+\epsilon)}$; let $Y$ be the 
union of these balls. Using a variant of the Kneser-Poulsen conjecture\footnote{Despite retaining its original name, this 
``conjecture'' is actually a proven theorem.} for continuous contractions, see \cite{Bez08,Csi06}, we get that $\mu(X)\leq \mu(Y)$ and therefore
$$\mu((\frac{1}{\ell} E)^{+1})\leq \mu(X)\leq \mu(Y)\leq \mu(E^{+(1+\epsilon)}).$$

Taking the limit as $\epsilon$ goes to $0$, $\mu((\frac{1}{\ell} E)^{+1})\leq \mu(E^{+1})$.

Now we can prove the lemma. Scaling $E^{+\ell}$ down by $\ell^{-1}$ yields $(\frac{1}{\ell} E)^{+1}$. Hence
$$\mu(E^{+\ell}) = \ell^d \mu ( (\frac{1}{\ell} E)^{+1}) \le \ell^d \mu(E^{+1}).$$
\end{proof}

The next lemma ensures that, given a patch of $E \cap\Lambda$, every patch of 
$\Lambda$ with translated support $E-x$ has ``approximately'' the same number of points,
the difference in the number of points being governed by $\mu(E^{+1})$ and the length of the shift.

\begin{lem} \label{lem:symdiff}
Let $\Lambda$ be a Delone set with parameters $r>0$ and $R>0$. 
Let $\ell \ge r$. Then there is $q>0$ such that, for every $x \in \R^d$ with
$\|x\| \le \ell$ and for every region $E$,
\[ | \# \big( (E+x) \cap \Lambda \big) - \#\big( E \cap \Lambda \big)| \le 
q \ell^d \mu(E^{+1}). \]
Here the constant $q$ depends on the Delone set $\Lambda$ and the dimension $d$.
\end{lem}
\begin{proof}
Let $\ell \ge r$ and $\|x\| \le \ell$. 
\begin{align*} 
& | \# \big( (E+x) \cap \Lambda \big) - \#\big( E \cap \Lambda \big) | \\
& \le  
\#\Big( \Big( \big( (E+x) \setminus E \big) \cup \big( E \setminus
(E+x) \big) \Big) \cap \Lambda \Big)\\
& \le  \# \Big( \big( (E+x) \setminus E \big) \cap \Lambda \Big) 
+ \# \Big( \big( E \setminus (E+x) \big) \cap \Lambda \Big)\\
& \stackrel{(x\le \ell)}{\le} \# \big( (E+x)^{+\ell} \cap 
\Lambda \big) + \# \big( E^{+\ell} \cap \Lambda \big)\\
& \stackrel{\mbox{\scriptsize (Lem. \ref{lem:r-R-estim})}}{\le}  
\frac{\eta'}{r^d} \Big( \mu \big( (E+x)^{+\ell} \big) + \mu \big( ((E+x)^{+\ell})^{+r} \big) \Big) 
+ \frac{\eta'}{r^d} \Big( \mu \big( E^{+\ell} \big) + \mu \big( (E^{+\ell})^{+r} \big) \Big) \\
& \stackrel{\mbox{\scriptsize (Lem. \ref{lem:r+l})}}{\le}
  2\frac{\eta'}{r^d} \big( \mu(E^{+\ell}) + \mu(E^{+(\ell+r)}) \big)
  \le 2\frac{\eta'}{r^d} 
\big( \mu(E^{+\ell}) + \mu(E^{+(\ell+\ell)}) \big)\\ 
& \stackrel{\mbox{\scriptsize (Lem. \ref{lem:thickboundary}})}{\le}
2\frac{\eta'}{r^d} (\ell^d + (\ell+\ell)^d) \mu(E^{+1}) \le
2 \frac{\eta'}{r^d} (1+2^d)\ell^d \mu(E^{+1}). 
\end{align*}
With $q=2\frac{\eta'}{r^d}(1+2^d)$ and $\eta'$ the constant from Lemma 
\ref{lem:r-R-estim} the claim follows.
\end{proof}

The next lemma is the key to proving Theorem \ref{thm:one-or-many} and particularly to the construction of ``deviant'' patches of $\Lambda$.

\begin{lem} \label{lem:E-x-y}
Let $\Lambda \subset \R^d$ be a repetitive Delone set of FLC with 
$\dens(\Lambda)=\varrho$ such that $\Lambda$ is not bde to any 
lattice in $\R^d$. Let $c>0$ and $\ell>0$ be given. Then there exists
a bounded measurable set $E$ such that for all $x\in \R^d$ with $\|x\| 
\le \ell$ holds: 
\[ | \# \big( (E-x) \cap \Lambda \big) - \varrho \mu(E) | > c \mu (E^{+1})  \]
\end{lem}
\begin{proof}
Let $q$ be as in Lemma \ref{lem:symdiff}. By Theorem \ref{thm:lacz} there 
is $E$ such that 
\begin{equation} \label{eq:wiggle}
|\#(E \cap \Lambda) - \varrho \mu(E) | > (c+q \ell^d) \mu(E^{+1}). 
\end{equation}
By Lemma \ref{lem:symdiff} we have that $|\# (E \cap \Lambda) - \# ( (E-x) \cap 
\Lambda ) | \le q \ell^d \mu(E^{+1})$. Replacing $\#(E \cap \Lambda)$ by
$\#((E-x) \cap \Lambda)$ in \eqref{eq:wiggle} changes the left hand side by 
less than $q \ell^d \mu(E^{+1})$. This yields the claim.
\end{proof}

\begin{lem} \label{lem:vanhove}
Let $\Lambda$ be a Delone set in $\R^d$ with $\dens(\Lambda)=\varrho$.
Let $(E_i)_i$ be a sequence of bounded measurable subsets of $\R^d$ that
violate the condition in Theorem \ref{thm:lacz}. That is, let the sequence $(c_i)_i$ be such
that $\lim\limits_{i \to \infty} c_i = \infty$, and for each $i$ 
\begin{equation} \label{eq:deviant}
 | \#(E_i \cap \Lambda) - \varrho \mu(E_i)| > c_i \mu(E^{+1}). 
\end{equation}
Then $(E_i)_i$ is a van Hove sequence.
\end{lem}
\begin{proof}
Inequality \eqref{eq:deviant} is equivalent to
\begin{equation} \label{eq:lim0} \frac{1}{c_i} \Big| \frac{\#(E_i \cap \Lambda)}{\mu(E_i)} - \varrho \Big|
> \frac{\mu(E_i^{+1})}{\mu(E_i)}.
\end{equation}

Using Lemma \ref{lem:r-R-estim} we get 
\[\frac{\eta}{R^d}-\frac{\mu(E_i^{+R})}{\mu(E_i)}\leq \frac{\#(E_i \cap \Lambda)}{\mu(E_i)}\leq \frac{\eta'}{r^d}+\frac{\mu(E_i^{+r})}{\mu(E_i)}.\]
From Lemma \ref{lem:thickboundary} we get upper bounds for $\mu(E_i^{+r})$ and $\mu(E_i^{+R})$ in terms of $\mu(E_i^{+1})$ (there is the trivial bound $\mu(E_i^{+1})$ if $r\leq 1$ or $R\leq 1$). Combining these estimates we get the following inequality for some positive constants $\alpha$ and $\beta$
\[\Big| \frac{\#(E_i \cap \Lambda)}{\mu(E_i)} - \varrho \Big|\leq \alpha+\beta  \frac{\mu(E_i^{+1})}{\mu(E_i)}. \]
Once we plug this in inequality \eqref{eq:lim0}, we get
\[\frac{1}{c_i} \left(  \alpha+\beta  \frac{\mu(E_i^{+1})}{\mu(E_i)}\right)
> \frac{\mu(E_i^{+1})}{\mu(E_i)}.\]

Considering the limits of both sides yields $\lim\limits_{i \to \infty} 
\frac{\mu(E^{+1})}{\mu(E_i)} = 0$. In order to show  $\lim\limits_{i \to \infty} 
\frac{\mu(E^{+\varepsilon})}{\mu(E_i)} = 0$ for all $\varepsilon >0$, we consider two cases. If $\varepsilon \le 1$, then
\[ \lim\limits_{i \to \infty} \frac{\mu(E_i^{+\varepsilon})}{\mu(E_i)} \le 
\lim\limits_{i \to \infty}\frac{\mu(E_i^{+1})}{\mu(E_i)} = 0. \] 
If $\varepsilon > 1$, then we use the estimate from Lemma \ref{lem:thickboundary} to get 
\[ \lim\limits_{i \to \infty} \frac{\mu(E_i^{+\varepsilon})}{\mu(E_i)} \le 
\lim\limits_{i \to \infty}\frac{\varepsilon^d\mu(E_i^{+1})}{\mu(E_i)} = 0. \] 
\end{proof}

Frequently a van Hove sequence $(E_i)_i$ is defined by requiring another
property in addition to \eqref{eq:def-vanhove}: namely, that the $E_i$ 
exhaust the entire space. For instance, the sequence $(B_i(x_i))_i$, 
with $x_i = (i,0, \cdots, 0)^T$, fulfills \eqref{eq:def-vanhove}, but the (closure of the) union of the balls is only 
the half-space $\R^{\ge 0} \times \R^{d-1}$. The next result ensures that we do not need this
further requirement in our context.

\begin{prop} \label{prop:balls}
Let $(E_i)_i$ be a van Hove sequence. Then for each $R>0$ there are
$i \ge 0, t_i \in \R^d$ such that $E_i$ contains a ball $B_R(t_i)$. 
\end{prop}
\begin{proof}
Assume the contrary; that is, suppose there is $R>0$ such that for
all $t \in \R^d, i \ge 0$ holds: $B_R(t_i) \not \subseteq E_i$. This implies
that for all $t,i$ we have that the distance between $t$ and $\partial E_i$
is not greater than $R$. Consequently, for all $i$ we have $E_i \setminus
(E_i)^{+R} = \varnothing$. Since $E_i \ne \varnothing$ this implies
that for all $i>0$ holds $\mu(E_i) < \mu((E_i)^{+R})$, contradicting
the van Hove property \eqref{eq:def-vanhove}. 
\end{proof}

Now we can state the proof of our main result. 
\begin{proof}[Proof of Theorem \ref{thm:one-or-many}]
Let $\Lambda$ be a Delone set in $\R^d$ with $\dens(\Lambda)=\varrho$
such that $\Lambda$ is not bde to any lattice in $\R^d$. 
We choose an infinite word $u = u_1 u_2 \cdots$ over the alphabet 
$\{D,N\}$, and $(c_i)_i$ such that $\lim\limits_{i \to \infty} c_i = \infty$.

For $c>0$, we call a bounded measurable set $E$ $c$-\emph{deviant} if
\[ | \# (E \cap \Lambda) - \varrho \mu(E)| > c \mu(E^{+1}), \] 
otherwise we call $E$ $c$-\emph{normal}.

Let $\ell_1$ be arbitrary. By Theorem \ref{thm:lacz} there is a compact set 
$D_1'$ that is $c_1$-deviant. By Lemma \ref{lem:exceptional} there is set $N_1':=D_1'-y_1$ 
such that the {\em discrepancy} $\#(N_1' \cap \Lambda) - \varrho \mu(N_1')$ of $N_1'$ 
is zero or has the opposite sign as the discrepancy of $D_1'$. 
($N_1'$ may or may not be $c$-normal for any prescribed $c$, but in all cases the discrepancy of $N_1'$ differs greatly from that of 
$D_1'$, which is what we actually need.) Pick a point $x_1 \in D_1'$.
If the first letter $u_1$ of $u$ is $D$ then let $P_1:=(D'_1 \cap \Lambda) -x_1 = 
(D_1' - x_1) \cap (\Lambda-x_1)$, which is a patch in $\Lambda -x_1$. 
Otherwise let $P_1:=( N_1' \cap \Lambda) +y_1-x_1 = (N'_1+y_1-x_1) \cap 
(\Lambda + y_1-x_1)$, which is a patch in $\Lambda + y_1 - x_1$. 
The shape of support of $P_1$ is the same in both cases, but the underlying 
patch of the Delone set is different.

Choose $\ell_2$ such that  $\ell_2 > 2 \max \{ R_{rep}(D_1'\cap \Lambda), 
R_{rep}(N_1'\cap \Lambda)\}$, where $R_{rep}(P)$ denotes the repetitivity radius
of the patch $P$ (compare Section \ref{sec:intro}). 
By Lemma \ref{lem:E-x-y} there is $D_2'$ such that  $D_2'-x$ is 
$c_2$-deviant for all $x$ with $\|x\|\le\ell_2$. By Lemma \ref{lem:exceptional}
there is $y$ such that $N_2' = D_2'-y$ has the opposite discrepancy as $D_2'$ (or this discrepancy is 0). 
Let $B_s(t_1)$ be the largest ball contained in $D_2'$.

By repetitivity, both $D_2' \cap \Lambda$ and $N_2'\cap \Lambda$ contain 
translates of both $D_1'\cap \Lambda$ and $N_1' \cap \Lambda$ within distance $\frac{\ell_2}{2}$ of the center $t_1$
of $B_s(t_1)$, say with the points $x_1 \in D_1'$ and $y_1 \in N_1'$ corresponding to $x_{D2}$ and $x_{N2}$ in $D_2'$.
If $u_1=D$ we take $x_2=x_{D2}$ and if $u_1=N$ we take $x_2=x_{N2}$. Either way, define $D_2 = D_2'-x_2$, viewed as 
a pattern in $\Lambda - x_2$. That is, $D_2$ is a translate (of both support and Delone set) of $D_2'$ with a copy of 
$P_1$ close to the center. We similarly create $N_2$ as a translate of $N_2'$ that likewise extends $P_1$. 
Finally, we pick $P_2$ to be either $D_2$ or $N_2$, depending on whether $u_2$ is $D$ or $N$. 

The supports of the two possible choices of $P_2$ are not identical, since the relative positions of copies of $P_1$ in
$D_2'$ and $N_2'$ are not the same. However, their supports differ by translation by less than $\ell_2$. A translate
of $D_2$ that has the same support as $N_2$ still has a discrepancy greater in magnitude than $c_2\mu(D_2^{+1})$, 
while the discrepancy of $N_2$ has a discrepancy of the opposite sign. 

%Clearly the supports
%of the patches $E_2 \cap \Lambda$ and $((E_2-y) \cap \Lambda) + y$ are identical.

%In order to make the supports of the two translates of $P_1$ identical we 
%can choose $x, x'$ with $\|x\|\le \ell_2/2$ and $\|y\|\le \ell_2/2$ such that 
%$(E_2 \cap \Lambda)-x$ and $((E_2-y) \cap \Lambda)+y-x'$ have
%$P_1$ at identical positions. Now the supports $E_2-x$ and $E_2+y-y-x'$ are
%not longer equal, they differ by a translation $z$ with $\|z\| \le \ell_2$. 
%By Lemma \ref{lem:E-x-y} the patch $((E_2-z) \cap \Lambda)-x$ is still deviant
%wrt $c_2$ (and of course the second patch is still normal wrt $c_2$). 
%Hence let $P_2 = ((E_2-z) \cap \Lambda)-x$ if $u_2=D$, and let $P_2 = 
%((E_2-y) \cap \Lambda)+y-x'$ if $u_2=N$. By construction, the two patches
%have the same support, and the point count yields
%\[ \# (((E_2-z) \cap \Lambda)-x)  - \# (((E_2-y) \cap \Lambda)+y-x') > c_2 \mu(E^{+1}). \]

Now iterate: let $\ell_{i+1}$ be such that $\ell_{i+1} > 2 \max \{ r_{rep}(D_i'
\cap \Lambda), r_{rep}(N_i' \cap \Lambda)\}$ and continue as above, finding regions 
$D_{i+1}'$ and $N_{i+1}'$ in $\Lambda$ such that $D_{i+1}'$ and all its translations by at most $\ell_{i+1}$ are $c_{i+1}$-deviant, 
$N_{i+1}'$ has a discrepancy of the opposite sign as $D_{i+1}'$, and such that both $D_{i+1}'$ and 
$N_{i+1}'$ contain copies of $D_i'$ and $N_i'$ within a distance $\ell_{i+1}/2$ of the center of 
a large sphere (as in Proposition \ref{prop:balls}). We then defined translates 
$D_{i+1}$ and $N_{i+1}$ of $D_{i+1}'$ and $N_{i+1}'$, viewed as regions 
in translates of $\Lambda$, such that these pattern extend $P_i$. Finally we pick 
$P_{i+1}$ to be $D_{i+1}$ or $N_{i+1}$ depending on the $(i+1)$st letter of $u$. 

For any
word $u \in \{D,N\}^{\N}$, this procedure gives a nested sequence of patches $(P_i)_i$ (i.e. $P_i \subset 
P_{i+1}$). By Lemma \ref{lem:vanhove} the supports of
the patches $P_i$ are a van Hove sequence. By Proposition \ref{prop:balls}
these van Hove sequences contain arbitrary large balls. Since we have chosen
the balls $B_s(t_i)$ in each step to be the largest possible, their diameters
$s$ tend to infinity. By the closedness of $\X_{\Lambda}$ under translations 
we can assume without loss of generality 
that these balls are centered at 0. The union of the patches $P_i$ is then a Delone set 
$\Lambda_u \in \X_\Lambda$. If $u,u' \in \{D,N\}^{\N}$ 
differ in infinitely many letters, then we look at the points of disagreement and compare 
$N_i$ in one Delone set to a translate by less than $\ell_i$ 
of $D_i$ in the other. Applying Theorem \ref{thm:l-not-bd}, we see that 
$\Lambda_u \not\bd \Lambda_{u'}$ because the translate by at most $\ell_i$ of $D_i$ is still $c_i$-deviant. There are $2^{\aleph_0}$ elements of 
$\{D,N\}^\N$ and only countably many elements of each bde class, so there are
$2^{\aleph_0}$ distinct bde classes among the $\{\Lambda_u\}$ and there are at least $2^{\aleph_0}$ bde classes
in $\X_\Lambda$.

However, Delone sets that differ only by a translation are bde and 
$\X_\Lambda$ consists of only $2^{\aleph_0}$ translational orbits, so $\X_\Lambda$ can have at most $2^{\aleph_0}$ bde
classes. Thus the cardinality of the bde classes in $\X_\Lambda$ is exactly $2^{\aleph_0}$.
\end{proof}

\begin{rem} In order to maintain a strict separation between our results and those of Smilansky and Solomon \cite{SmSo}, 
we stated and proved Theorem \ref{thm:one-or-many} in the form in which we originally developed it, including the assumption
that an overall density exists. However, it is extremely easy to drop that assumption. If there is no overall density,
then there is an upper density $\rho_+$ and a lower density $\rho_-$. Instead of looking for regions 
that are rich or poor relative to ``the'' overall density, we can look instead for regions that are rich or poor relative to
an arbitrary intermediate density $\rho \in (\rho_-,\rho_+)$. For instance, we can look for Euclidean balls of large radius 
whose densities are greater than $(\rho_++\rho)/2$ or are less than $(\rho_-+\rho)/2$. The rest of the proof that there
exist uncountably many bde classes proceeds exactly as before. 
\end{rem}

\section*{Acknowledgements}
It is a pleasure to thank Philipp Gohlke and Yasushi Nagai for providing
the proof of Proposition \ref{prop:balls} (actually, two proofs).

The authors are also extremely thankful to Yaar Solomon and Alan Haynes for useful comments.

\end{document}